\documentclass[10pt]{amsart}
\usepackage[matrix,arrow]{xy}
\usepackage{amssymb}
\usepackage{amsmath}
\usepackage{amsfonts}
\usepackage{epsfig}
\usepackage{color}
\usepackage{epstopdf}
\usepackage{graphicx}
\usepackage{amsthm}
\usepackage{enumerate}
\usepackage[mathscr]{eucal}
\usepackage{verbatim}
\usepackage{bm}

\usepackage[bookmarks=true,hyperindex,pdftex,colorlinks,citecolor=blue,
linkcolor=blue,urlcolor=blue]{hyperref}
\usepackage{tikz}
\usetikzlibrary{matrix,arrows.meta}

\theoremstyle{plain}
\newtheorem{theorem}{Theorem}[section]
\newtheorem{cor}[theorem]{Corollary}
\newtheorem{prop}[theorem]{Proposition}
\newtheorem{lemma}[theorem]{Lemma}

\theoremstyle{definition}

\newtheorem{example}[theorem]{Example}

\newtheorem{question}[theorem]{Question}

\newtheorem{remark}[theorem]{Remark}
\newtheorem{definition}[theorem]{Definition}

\renewcommand{\subset}{\subseteq}

\title[Weak minimizing property on pairs of classical Banach spaces]
{Weak minimizing property on pairs of classical Banach spaces}

\author[M.~Han]{Manwook Han}
\address[M.~Han]{Department of Mathematics, Chungbuk National University, Cheongju, Chungbuk 28644, Republic of Korea\newline
	\href{https://orcid.org/0009-0000-7250-1814}{ORCID: \texttt{0009-0000-7250-1814}  }}
\email{\texttt{mwhan0828@gmail.com}}


\begin{document}

\begin{abstract}
We investigate the minimum modulus analogue of the weak maximizing property, termed the \emph{weak minimizing property}. We establish that the pairs $(\ell_p, L^p[0, 1])$ for $2 \leq p < \infty$ and $(\ell_s \oplus_q \ell_q, \ell_r \oplus_p \ell_p)$ for $1 < p \leq r\leq s \leq q < \infty$ satisfy the weak minimizing property. Conversely, we prove that the pairs $(\ell_1, \ell_p)$, $(\ell_1, c_0)$, $(\ell_1, \ell_1)$ and $(c_0, \ell_p)$ fail to satisfy the weak minimizing property.
\end{abstract}

\maketitle

\section{Introduction}

The seminal work of Bishop and Phelps \cite{BP} demonstrated that every bounded linear functional can be approximated by norm-attaining functionals. Subsequently, Lindenstrauss revealed that the Bishop-Phelps theorem no longer holds for vector-valued bounded linear operators \cite{Lind}. At the same time, he showed that it does hold when the domain is a reflexive Banach space, and posed two significant questions as follows.
\begin{enumerate}
    \item[A.] For which domain spaces does the vector-valued version of the Bishop-Phelps theorem always hold?
    \item[B.] For which range spaces does the vector-valued version of the Bishop-Phelps theorem always hold?
\end{enumerate}
These questions spurred the interest of many mathematicians in the norm-attaining operator theory, making it a fruitful topic in functional analysis for the subsequent decades.

While considerable progress has been made on question A, notably by Huff and Bourgain in \cite{Huff, Bour}, advancements concerning question B have been comparatively modest. Gowers's result, demonstrating that the classical sequence space $\ell_p$ ($1<p<\infty$) is not a range space described in question B \cite{Gow}, is particularly noteworthy. However, it remains unknown whether every finite-dimensional Banach space is that kind of range space or not.
Beyond linear operators, the theory of norm-attaining maps has been extended to various settings --- for instance, to bilinear mappings \cite{AAP, Choi}, homogeneous polynomials \cite{ArGM, AGM} and Lipschitz maps \cite{KMS, CCGMR}.

Another noteworthy direction of research concerns the \textbf{minimum modulus}.
Let $X$ and $Y$ be Banach spaces, and denote by $\mathcal{L}(X,Y)$ the space of bounded linear operators from $X$ to $Y$.
The minimum modulus of operator $T\in\mathcal{L}(X,Y)$ is defined as
$$m(T)=\inf\{\|Tx\|:x\in S_X\},$$
where $S_X$ denotes the unit sphere of $X$.
We say that $T$ \textbf{attains its minimum modulus} if there exists an element $x\in S_X$ such that $\|Tx\|=m(T)$.
In this case, $T$ is said to be a \textbf{minimum modulus-attaining operator}, in analogy with norm-attaining operators.

Chakraborty demonstrated that if the domain of a bounded linear operator has the Radon-Nikod\'ym property, then the operator can be approximated by minimum modulus-attaining operators \cite{Cha}.
In the same article, the author stated the \textbf{weak minimizing property} (WmP) which is the minimum modulus analogue of the \textbf{weak maximizing property} (WMP) introduced in \cite{AGPT}. Recall that a pair of Banach spaces $(X,Y)$ has the WMP if for every bounded linear operator with a non-weakly null maximizing sequence attains its norm. Here, a maximizing sequence $(x_n)_n$ of operator $T$ means that $\|Tx_n\|\to\|T\|$ and contained in $S_X$.

The literature on the WMP is now substantial and contains several notable contributions in norm-attaining operator theory. Pellegrino and Teixeira produced the first nontrivial examples of pairs with the WMP, namely $(\ell_p,\ell_q)$ for $1<p<\infty$ and $1\leq q<\infty$ \cite{PT}. Aron, Garc\'ia, Pellegrino and Teixeira proved that every pair of Hilbert spaces has the WMP \cite{AGPT}. Dantas, Jung, and Mart\'inez-Cervantes showed that a Banach space $Y$ has the Schur property if and only if the pair $(X,Y)$ has the WMP for every reflexive Banach space $X$ \cite{DJM}. Recently, a sufficient condition for the WMP has been obtained in \cite{HK}. In particular, it is proved that $(\ell_p,L_p[0,1])$ for $2\leq p<\infty$ has the WMP.

The WMP can be viewed as a property stronger than the denseness of norm-attaining operators, since it implies that the set of operators which do not attain their norm is contained in a nowhere-dense subset of $\mathcal{L}(X,Y)$.
In fact, if a pair of Banach spaces $(X,Y)$ satisfies the WMP, then it satisfies the compact perturbation property (CPP).
Note that a pair $(X,Y)$ is said to have the CPP if for given operator $T\in\mathcal{L}(X,Y)$ that does not attain its norm it holds that
$$\inf_{K\in\mathcal{K}(X,Y)}\|T+K\|=\|T\|.$$
Here, $\mathcal{K}(X,Y)$ denotes the space of compact operators from $X$ to $Y$. 
On the other hand, the CPP does not imply the WMP in general \cite{GMA}.

The main aim of the present paper is to investigate the minimum modulus version of the WMP.
\begin{definition}
    Let $X$ and $Y$ be Banach spaces.
    \begin{itemize}
        \item For $T\in\mathcal{L}(X,Y)$, a sequence $(x_n)_n$ in $S_X$ is said to \textbf{minimize} $T$ if $\|Tx_n\|\to m(T)$. In this case, $(x_n)_n$ is called a \textbf{minimizing sequence} for $T$. 
        \item A pair $(X,Y)$ is said to have the \textbf{weak minimizing property} (WmP) if every bounded linear operator that admits a non-weakly null minimizing sequence attains its minimum modulus.
    \end{itemize}
\end{definition}

We first present the known results.

\begin{remark}\label{Cha}\cite{Cha} \leavevmode\vspace{0.3em} 
    \begin{itemize} 
        \item Every pair with the WmP also satisfies the \textbf{compact perturbation property for minimum modulus} (CPPm), that is, every non-minimum modulus-attaining operator satisfies 
        $$\sup_{K\in\mathcal{K}(X,Y)}m(T+K)=m(T).$$
        \item Every pair of Hilbert spaces satisfies the WmP. 
        \item For $1<p<\infty$ and $1\leq q<\infty$, $(\ell_p,\ell_q)$ has the WmP. 
        \item If either $X$ or $Y$ is finite-dimensional, then $(X,Y)$ has the WmP. 
    \end{itemize} 
\end{remark}

As the CPP implies the denseness of norm-attaining operators, the CPPm implies the denseness of minimum modulus-attaining operators. 
Indeed, if a pair $(X,Y)$ of Banach spaces satisfies the CPPm, then for every operator $T\in\mathcal{L}(X,Y)$, it holds that either $m(T)=0$ or $m(T)>0$. 
If $m(T)=0$, then for each $\varepsilon>0$, there exists $x\in S_X$ such that $\|Tx\|<\varepsilon$. By the Hahn-Banach theorem, we can find $x^*\in S_{X^*}$ such that $x^*(x)=1$. 
Hence, $\|x^*\otimes Tx\|<\varepsilon$, and consequently the operator $S=T-x^*\otimes Tx$ satisfies $\|T-S\|\leq\varepsilon$ and $Sx=0=m(S)$. 
If $m(T)>0$, then for each $\varepsilon>0$, there exist $x\in S_X$ and $x^*\in S_{X^*}$ such that $\|Tx\|<m(T)+\delta$ for $\delta:=\min\left\{\varepsilon,\frac{1}{2}m(T)\right\}$ and $x^*(x)=1$. Define $S=T-x^*\otimes Tx \frac{\delta}{\|Tx\|}$. Then $\|Sx\|<m(T)$ and therefore, by the CPPm, the operator $S$ attains its minimum modulus. Moreover, $\|T-S\|\leq\delta\leq\varepsilon$.

In Section \ref{section2}, we adapt the approach of \cite{HK} in a certain restricted framework to establish sufficient conditions for the WmP.
These conditions are then applied to identify new classes of pairs of Banach spaces satisfying the WmP.

In Section \ref{section3}, we construct several examples illustrating the failure of the WmP and the CPPm. 
We also show that, in general, the WMP does not imply the WmP.
Moreover, we present a simple example that possesses the CPPm but fails to satisfy the WmP.

\section{Weak minimizing property and property $(m)$}\label{section2}

Before turning to the main results, we recall a few notions.

\begin{definition}[\cite{KW,K1}]
    Let $X$ and $Y$ be Banach spaces. 
    \begin{itemize}
        \item An operator $T\in\mathcal{L}(X,Y)$ is said to satisfy \textbf{property $(M)$} if for all $x\in X$, $y\in Y$ with $\|y\|\leq\|x\|$ and for every weakly null sequence $(x_n)_n$ in $X$, one has
        $$\limsup_n\|y+Tx_n\|\leq\limsup_n\|x+x_n\|.$$
        \item The space $X$ is said to have \textbf{property $(M)$} if the identity operator $\mathrm{id}_X$ on $X$ satisfies property $(M)$.
        \item The pair $(X,Y)$ is said to satisfy \textbf{property $(M)$} if every operator $T\in\mathcal{L}(X,Y)$ with operator norm at most $1$ satisfies property $(M)$.
    \end{itemize}
\end{definition}

Property $(M)$ was introduced by Kalton and Werner as a necessary condition for $\mathcal{K}(X,Y)$ to be an $M$-ideal in $\mathcal{L}(X,Y)$ \cite{K1}. 
Note that the pair $(X,X)$ satisfies property $(M)$ if and only if $X$ itself has property $(M)$. 

The {Opial property} was introduced by Opial in the study of weak convergence of nonexpansive mappings \cite{Op}. Although it has a different origin from property $(M)$, it also concerns the asymptotic behavior of the norm. This concept has been extended to operators between Banach spaces, leading to {property $(O)$} \cite{HK}.

\begin{definition}
    Let $X$ and $Y$ be Banach spaces.
    \begin{itemize}
        \item An operator $T\in \mathcal{L}(X,Y)$ is said to satisfy \textbf{property $(O)$} if for every nonzero $x\in X$ and every weakly null sequence $(x_n)_n$ in $X$, one has
        $$\limsup_n \|Tx_n\| < \limsup_n \|x+x_n\|.$$
        \item The space $X$ is said to have the \textbf{Opial property} if the identity operator $\mathrm{id}_X$ on $X$ satisfies property $(O)$.
        \item The pair $(X,Y)$ is said to satisfy \textbf{property $(O)$} if every operator $T\in\mathcal{L}(X,Y)$ with operator norm at most $1$ satisfies property $(O)$.
    \end{itemize}
\end{definition}

Analogously, the pair $(X,X)$ satisfies property $(O)$ if and only if $X$ has the Opial property. 
It is worth mentioning that Van Dulst showed that every separable Banach space can be renormed to have the Opial property \cite{Dul}.

In \cite{HK}, the authors established the following connection between these properties and the WMP, and gave new examples with the WMP. The following theorem provides a primary motivation for the present study.

\begin{theorem}\label{HKthm}\cite{HK}
    For Banach spaces $X$ and $Y$, if $X$ is reflexive and the pair $(X,Y)$ satisfies properties $(M)$ and $(O)$ simultaneously, then $(X,Y)$ has the WMP.
\end{theorem}

Our objective in this section is to observe the analogous connection to find new examples of pairs of Banach spaces satisfying the WmP. To achieve this, we proceed to establish sufficient conditions for the WmP based on the notions of {properties $(m)$ and $(o)$}, and then apply these criteria to specific examples.

\begin{definition}
    Let $X$ and $Y$ be Banach spaces. 
    \begin{itemize}
        \item An operator $T\in\mathcal{L}(X,Y)$ is said to satisfy \textbf{property $(m)$} if for every $x\in X$, $y\in Y$ with $\|x\|\leq\|y\|$ and for every weakly null sequence $(x_n)_n$ in $X$, one has
        $$\limsup_n\|x+x_n\|\leq\limsup_n\|y+Tx_n\|.$$
        \item The pair $(X,Y)$ is said to satisfy \textbf{property $(m)$} if every operator $T\in\mathcal{L}(X,Y)$ with minimum modulus at least $1$ satisfies property $(m)$.
        \item An operator $T\in \mathcal{L}(X,Y)$ is said to satisfy \textbf{property $(o)$} if for every nonzero $y\in Y$ and every weakly null sequence $(x_n)_n$ in $X$, one has
        $$\limsup_n \|x_n\| < \limsup_n \|y+Tx_n\|.$$
        \item The pair $(X,Y)$ is said to satisfy \textbf{property $(o)$} if every operator $T\in\mathcal{L}(X,Y)$ with minimum modulus at least $1$ satisfies property $(o)$.
    \end{itemize}
\end{definition}

\begin{lemma}\label{strictm}
    Let $X$ and $Y$ be Banach spaces. For all $x\in X$, $y\in Y$ with $\|x\|<\|y\|$, for every operator $T\in\mathcal{L}(X,Y)$ with $m(T)\geq1$ and for every weakly null sequence $(x_n)_n$ in $X$, it satisfies
    $$\limsup_n\|x+x_n\|<\limsup_n\|y+Tx_n\|$$
    if and only if the pair $(X,Y)$ satisfies properties $(o)$ and $(m)$ simultaneously.
\end{lemma}
\begin{proof}
    The `only if' part is clear. For the `if' part, assume $(X,Y)$ satisfies both properties $(m)$ and $(o)$. 
    Let $T\in\mathcal{L}(X,Y)$ be an operator with $m(T)\geq 1$ and let $x\in X$, $y\in Y$ be elements with $\|x\|<\|y\|$, and $(x_n)_n$ be a weakly null sequence in $X$.
    For $\lambda<1$ such that $\frac{1}{\lambda}\|x\|<\|y\|$, it follows that
    $$\limsup_n \left\| \frac{1}{\lambda} x + x_n \right\| \leq \limsup_n \|y+Tx_n\|\quad \text{by property $(m)$},$$
    and
    $$\limsup_n \|x_n\| < \limsup_n \|y+Tx_n\| \quad \text{by property $(o)$.}$$
    Hence we have
    $$\limsup_n \|x+x_n\| \leq \lambda \limsup_n \left\|\frac{1}{\lambda}x+x_n\right\| + (1-\lambda) \limsup_n \|x_n\| < \limsup_n \|y+Tx_n\|,$$
    as required.
\end{proof}

We next establish the promised sufficient condition for the WmP.

\begin{theorem}\label{Main}
    If a pair $(X,Y)$ of Banach spaces has properties $(m)$ and $(o)$ simultaneously and $X$ is reflexive, then $(X,Y)$ has the WmP.
\end{theorem}

\begin{proof}
    For $T\in \mathcal{L}(X,Y)$, let $(x_n)_n \subset S_X$ be a non-weakly null minimizing sequence for $T$. Passing to a subsequence if necessary, assume that $x_n$ converges weakly to a nonzero element $x\in B_X$. If $m(T)=0$, then the sequence $(Tx_n)_n$ converges to $0$. 
    Since $Tx_n\to Tx$ in the weak topology, it follows that $Tx=0$. Hence,
    $$T\left(\frac{x}{\|x\|}\right)=m(T)=0.$$
    If $m(T)=1$, define a weakly null sequence $(w_n)_n$ by $w_n = x_n - x$ for each $n\in\mathbb{N}$. Suppose $T$ does not attain its minimum modulus.
    Then $\|Tx\| > \|x\|$. 
    Therefore,
    $$1 = \limsup_{n} \|Tx_n\| = \limsup_n \|Tx + Tw_n\| > \limsup_n \|x + w_n\| = \limsup_n \|x_n\| = 1,$$
    where the inequality follows from Lemma \ref{strictm}, leading to a contradiction.

    Lastly, if $m(T)\notin \{0,1\}$, set $S = \frac{T}{m(T)}$.
    By the previous argument, $S$ attains its minimum modulus, and hence $T$ also attains its minimum modulus.
\end{proof}

Next, we investigate the properties $(m)$ and $(o)$, to construct examples of pairs of Banach spaces satisfying the WmP.

\begin{prop}\label{BasicFacts}
For Banach spaces $X$ and $Y$,
    \begin{enumerate}
        \item[(1)] If $X$ and $Y$ have property $(M)$, then $(X,Y)$ has property $(m)$.
        \item[(2)] If $Y$ has the Opial property, then $(X,Y)$ has property $(o)$.
        \item[(3)] If $X$ has the Opial property and $Y$ has property $(M)$, then $(X,Y)$ has property $(o)$.
        \item[(4)] If $X$ has the Opial property and $(X,Y)$ has property $(m)$, then $(X,Y)$ has property $(o)$.
    \end{enumerate}
\end{prop}

\begin{proof}
(1) Suppose both $X$ and $Y$ have property $(M)$, and fix an operator $T\in\mathcal{L}(X,Y)$ with $m(T)\geq1$, elements $x\in X$ and $y\in Y$ with $\|x\|\leq \|y\|$ and a weakly null sequence $(x_n)_n$. 

If $m(T)=1$, for any $\varepsilon>0$, there exists $z\in X$ such that $\|x\|=\|z\|$ and $\|Tz\|\leq (1+\varepsilon)\|y\|$. Then we observe
$$\limsup_n \|x+x_n\| = \limsup_n \|z+x_n\| \leq \limsup_n \|Tz+Tx_n\| \leq \limsup_n \|(1+\varepsilon)y+Tx_n\|.$$
Since $\varepsilon$ is arbitrary, we have 
$$\limsup_n\|x+x_n\|\leq\limsup_n\|y+Tx_n\|.$$
If $m(T)>1$, put $L=\frac{T}{m(T)}$. Since $m(L)=1$, we see that
\begin{align*}
\limsup_n\|x+x_n\|&\leq\limsup_n\|y+Lx_n\|\\
&\leq \max\{ \limsup_n \|y+Tx_n\|, \limsup_n \|y-Tx_n\| \} \\
&= \limsup_n \|y+Tx_n\|.
\end{align*}
as desired.

(2) Suppose $Y$ has the Opial property. Fix an operator $T\in\mathcal{L}(X,Y)$ with $m(T)\geq1$, a nonzero element $y\in Y$ and a weakly null sequence $(x_n)_n$ in $X$. Then we see that
$$\limsup_n \|x_n\| \leq \limsup_n \|Tx_n\| < \limsup_n \|y+Tx_n\|.$$

(3) Suppose $X$ has the Opial property and $Y$ has property $(M)$. Fix an operator $T\in\mathcal{L}(X,Y)$ with $m(T)\geq1$, a nonzero element $y\in Y$ and a weakly null sequence $(x_n)_n$. Then there exists a nonzero element $x\in X$ with $\|Tx\|\leq \|y\|$, consequently
$$\limsup_n \|x_n\| < \limsup_n \|x+x_n\| \leq \limsup_n \|Tx+Tx_n\| \leq \limsup_n \|y+Tx_n\|.$$

(4) Suppose $X$ has the Opial property and the pair $(X,Y)$ has property $(m)$. For an operator $T\in\mathcal{L}(X,Y)$ with $m(T)\geq1$, a nonzero element $y\in Y$ and a weakly null sequence $(x_n)_n$, there is a nonzero element $x\in X$ with $\|x\|\leq \|y\|$. It holds that
$$\limsup_n \|x_n\| < \limsup_n \|x+x_n\| \leq \limsup_n \|y+Tx_n\|,$$
as desired.
\end{proof}

We now apply this general principle to classical sequence spaces. 
It is well known that for $1\leq p<\infty$, the space $\ell_p$ satisfies both the Opial property and property $(M)$, while $c_0$ satisfies property $(M)$.
In fact, this follows from the equalities
\begin{equation}\label{eq:lp}
    \limsup_n\|x+x_n\|=\left(\|x\|^p+\limsup_n\|x_n\|^p\right)^{\frac{1}{p}},\quad\text{if}~X=\ell_p,
\end{equation}
and
\begin{equation}\label{eq:c0}
\limsup_n\|x+x_n\|=\max\left\{\|x\|,\limsup_n\|x_n\|\right\},\quad\text{if}~X=c_0,
\end{equation}
for any $x\in X$ and any weakly null sequence $(x_n)_n$ in $X$.
We say that a Banach space $X$ has property $(m_p)$ if it satisfies (\ref{eq:lp}) for some $1\leq p<\infty$, and property $(m_\infty)$ if it satisfies (\ref{eq:c0}).

The following Corollary extends \cite[Theorem 2.9]{Cha}.
\begin{cor}
    Let $X$ be a reflexive Banach space and $Y$ a Banach space.
    If $X$ has property $(m_p)$ and $Y$ has property $(m_q)$ for some $1\leq p,q\leq\infty$ with at least one of $p$ and $q$ finite, then the pair $(X,Y)$ has the WmP. 
    In particular, let $(X_i)_i$ and $(Y_i)_i$ be sequences of finite-dimensional Banach spaces, and define
    $$X_r=\left[\bigoplus_{i=1}^\infty X_i\right]_{\ell_r},\quad Y_s=\left[\bigoplus_{i=1}^\infty Y_i\right]_{\ell_s},\quad\text{and}\quad Y_0=\left[\bigoplus_{i=1}^\infty Y_i\right]_{c_0}.$$
    Then both pairs $(X_r,Y_s)$ and $(X_r,Y_0)$ satisfy the WmP for $1<r,s<\infty$.
\end{cor}
\begin{proof}
    Note that the spaces $X_r$, $Y_s$ and $Y_0$ satisfy properties $(m_r)$, $(m_s)$ and $(m_\infty)$ respectively. 
    Hence, the result follows directly from Theorem \ref{Main} together with Proposition \ref{BasicFacts}.
\end{proof}

We next extend these ideas to settings involving $L_p$ spaces. In \cite[Proposition 2.5]{KW}, it was shown that $\mathcal{K}(\ell_p,L_p[0,1])$ is an M-ideal in $\mathcal{L}(\ell_p,L_p[0,1])$ for $2\leq p<\infty$ by establishing property $(M)$ for $(\ell_p,L_p[0,1])$. We now prove that this pair also satisfies property $(m)$, and hence has the WmP.

\begin{lemma}\label{lpminM}
    For $1<p<\infty$ and a Banach space $Y$, the pair $(\ell_p,Y)$ has property $(m)$ if and only if for every $y\in Y$ and every operator $T\in\mathcal{L}(\ell_p,Y)$ it holds that
    $$\limsup_n \|y+Te_n\| \geq \left( \|y\|^p + m(T)^p \right)^{\frac{1}{p}},$$    
    where $(e_n)$ denotes the canonical basis of $\ell_p$.
\end{lemma}

\begin{proof}
The `only if' part is straightforward.
Indeed, if the pair $(\ell_p,Y)$ satisfies property $(m)$, then for given $y\in Y$ and $T\in\mathcal{L}(\ell_p,Y)$, we may assume $m(T)\neq0$.
Then we have
$$m(T)\cdot\limsup_n\left\|\frac{y}{m(T)}+\frac{T}{m(T)}e_n\right\|\geq m(T)\cdot\limsup_n\left\|\frac{\|y\|}{m(T)}e_1+e_n\right\|=\left(\|y\|^p+m(T)^p\right)^{\frac{1}{p}}.$$

For the converse, assume the pair $(\ell_p,Y)$ does not have property $(m)$.
Then there exist $x\in \ell_p$, $y\in Y$ with $\|x\|\leq\|y\|$, an operator $T\in\mathcal{L}(\ell_p,Y)$ with $m(T)\geq1$ and a weakly null sequence $(x_n)_n$ in $\ell_p$ such that
$$\limsup_n \|x+x_n\| > \limsup_n \|y+Tx_n\|.$$
Passing to a subsequence if necessary, we can assume that the limit $\lim_n \|x_n\|$ exists.
Moreover, this limit cannot be zero, hence, after normalization, we may assume that $\lim_n\|x_n\|=1$.

For any $\varepsilon>0$, there exists a subsequence $(x_{n_i})_i$ of $(x_n)_n$ and a $(1+\varepsilon)$-almost isometric operator $\Phi:\ell_p\to\overline{\text{span}}\{x_{n_i}\}$ such that $\Phi(e_i)=x_{n_i}$. Then we have
\begin{align*}
\limsup_n \|y+Tx_n\| &\geq \limsup_i \|y+Tx_{n_i}\|\\
&= \limsup_i \|y+T\Phi e_i\|\\
&\geq \left( \|y\|^p + m(T\Phi)^p \right)^{\frac{1}{p}} \geq \left( \|y\|^p + (1+\varepsilon)^{-p} \right)^{\frac{1}{p}}.
\end{align*}
Taking $\varepsilon\to0$, we obtain 
$$\limsup_n \|y+Tx_n\|\geq \left( \|y\|^p + 1\right)^{\frac{1}{p}}.$$
However, by assumption, we have
$$\left( \|x\|^p + 1\right)^{\frac{1}{p}} = \limsup_n \|x+x_n\| > \limsup_n \|y+Tx_n\| \geq \left( \|y\|^p + 1\right)^{\frac{1}{p}},$$
which leads to a contradiction.
\end{proof}

\begin{example}\label{NewEx1}
For $2\leq p<\infty$, the pair $(\ell_p,L_p[0,1])$ has property $(m)$ and therefore has the WmP.
\end{example}

\begin{proof}
If $p=2$, then both the domain and the range are Hilbert spaces, and hence the WmP follows from \cite{Cha}. On the other hand, note that for every Hilbert space $H$, it holds that $\mathcal{K}(H)$ is an M-ideal in $\mathcal{L}(H)$ \cite{AE}. In particular, $H$ has property $(M)$ \cite[VI. Theorem 4.17]{HWW}. Consequently by Proposition \ref{BasicFacts}, the pair $(\ell_2,L_2[0,1])$ has property $(m)$.

Suppose $p>2$. 
Let $T:\ell_p\to L_p[0,1]$ be a given operator, and let $y\in L_p[0,1]$. 
Since the sequence $(Te_n)_n$ tends to zero in measure (see, for instance, \cite[Proposition 6.5]{KW}), we have
$$\limsup_n \|y+Te_n\| = \left( \|y\|^p + \limsup_n \|Te_n\|^p \right)^{\frac{1}{p}} \geq \left( \|y\|^p + m(T)^p \right)^{\frac{1}{p}}.$$
Hence, the pair $(\ell_p,L_p[0,1])$ satisfies the WmP by Theorem \ref{Main}, Proposition \ref{BasicFacts}, and Lemma \ref{lpminM}.
\end{proof}

Next, we explore the relationships between properties $(m)$ and $(o)$ and the moduli of asymptotic uniform convexity and smoothness.

The \textbf{modulus of asymptotic uniform convexity} $\bar{\delta}_X:[0,\infty)\to[0,\infty)$ of a Banach space $X$ is the function defined by
$$\bar{\delta}_X(t) = \inf_{x\in S_X} \bar{\delta}_X(x,t),$$
where
$$\bar{\delta}_X(x,t) = \sup_{\dim(X/Y)<\infty} \inf_{\substack{y\in Y\\ \|y\|\geq t}} \|x+y\| - 1.$$

Similarly, the \textbf{modulus of asymptotic uniform smoothness} $\bar{\rho}_X:[0,\infty)\to[0,\infty)$ of a Banach space $X$ is given by
$$\bar{\rho}_X(t) = \sup_{x\in S_X} \bar{\rho}_X(x,t),$$
where
$$\bar{\rho}_X(x,t) = \inf_{\dim(X/Y)<\infty} \sup_{\substack{y\in Y\\ \|y\|\leq t}} \|x+y\| - 1.$$

We recall the following inequalities (see, for example, \cite[Remark 1]{HK}), which directly reveals the connection between properties $(m)$ and $(o)$. 
For every $x \in X$ and every weakly null sequence $(x_n)_n$ in $X$, we have
\begin{align*}
    \|x\|+\|x\|\bar{\delta}_X\left(\frac{\liminf_{n}\|x_n\|}{\|x\|}\right) &\leq \liminf_{n} \|x + x_n\|\\
    &\leq \limsup_n\|x+x_n\| \leq \|x\| + \|x\|\, \bar{\rho}_X\left( \frac{\limsup_{n} \|x_n\|}{\|x\|} \right).
\end{align*}

\begin{theorem}\label{GPanalyze}
For Banach spaces $X$ and $Y$, if $\bar{\rho}_X \leq \bar{\delta}_Y$, then the pair $(X,Y)$ has property $(m)$.
\end{theorem}

\begin{proof}
Let $T\in \mathcal{L}(X,Y)$ satisfy $m(T)\geq1$, $x\in X$ and $y\in Y$ satisfy $\|x\|\leq\|y\|$, and let $(x_n)_n$ be a weakly null sequence in $X$.

If $\|x\|=\|y\|=0$, then
$$\limsup_n \|x+x_n\| = \limsup_n \|x_n\| \leq \limsup_n \|Tx_n\| = \limsup_n \|y+Tx_n\|.$$
Suppose $\|x\|=\|y\|\neq0$. There is a subsequence $(x_{n_i})_i$ of $(x_n)_n$ satisfying
 $$\limsup_n \|x+x_n\| = \lim_i \|x+x_{n_i}\|,\quad \text{and}\quad \limsup_i\|x_{n_i}\|=\lim_i\|x_{n_i}\|.$$

Then we have
\begin{align*}
\limsup_n \|x+x_n\| &=\lim_i\|x+x_{n_i}\|\\
&\leq \|x\| + \|x\| \bar{\rho}_X\left( \frac{\limsup_i \|x_{n_i}\|}{\|x\|} \right)\\
&\leq \|y\| + \|y\| \bar{\rho}_X\left( \frac{\lim_i \|x_{n_i}\|}{\|y\|} \right)\\
&\leq \|y\| + \|y\| \bar{\delta}_Y\left( \frac{\lim_i \|x_{n_i}\|}{\|y\|} \right)\\
&\leq \|y\| + \|y\| \bar{\delta}_Y\left( \frac{\liminf_i \|Tx_{n_i}\|}{\|y\|} \right)\\
&\leq \liminf_i \|y+Tx_{n_i}\| \leq \limsup_n \|y+Tx_n\|.
\end{align*}

If $\|x\|<\|y\|$, there exists $v\in X$ with $\|v\|=\|y\|$ and $0\leq\lambda\leq1$ such that $x=\lambda v+(1-\lambda)(-v)$. Then we have
$$\limsup_n \|x+x_n\| \leq \max\{ \limsup_n \|v+x_n\|, \limsup_n \|-v+x_n\| \} \leq \limsup_n \|y+Tx_n\|.$$
Hence the desired inequality holds, which completes the proof.
\end{proof}

From Theorem \ref{Main}, Proposition \ref{BasicFacts}, Theorem \ref{GPanalyze}, and since \cite[Theorem 17]{HK} shows that if a Banach space $X$ satisfies $\bar{\delta}_X(x,t)\geq t-1$ for every $t>1$ and $x\in S_X$, then it has property $(O)$. Therefore we immediately obtain the following.

\begin{cor}\label{WmPmoduli}
If a pair $(X,Y)$ of Banach spaces satisfies the following conditions
\begin{itemize}
    \item[(i)] $\bar{\rho}_X\leq \bar{\delta}_Y$,
    \item[(ii)] either for every $t\geq 1$,
    $$\bar{\delta}_X(x,t)>t-1,\quad \text{for all}~ x\in S_X$$
    or for every $t\geq 1$,
    $$\bar{\delta}_Y(y,t)>t-1\quad \text{for all}~ y\in S_Y$$
    \item[(iii)] $X$ is reflexive,
\end{itemize}
then $(X,Y)$ has the WmP.
\end{cor}

We now present specific examples based on this result.
According to \cite[Theorem 2.12]{DJM}, for $1<p\leq s\leq q<\infty$, we have
\begin{equation}\label{lslplqeqn}
\bar{\rho}_{\ell_s\oplus_q\ell_q} = \bar{\delta}_{\ell_s\oplus_p\ell_p} = (1+t^s)^{\frac{1}{s}}-1.
\end{equation}
This leads to the following conclusion.

\begin{example}\label{NewEx2}
For $1<p\leq r\leq s\leq q<\infty$, the pair $(\ell_s\oplus_q\ell_q,\ell_r\oplus_p\ell_p)$ has property $(m)$. In particular, it has the WmP.
\end{example}
\begin{proof}
    By equation (\ref{lslplqeqn}) and the monotonicity of $\ell_p$-norm, we observe that
    $$\bar{\rho}_{\ell_s\oplus_q\ell_q}=(1+t^s)^{\frac{1}{s}}-1\leq(1+t^r)^{\frac{1}{r}}-1=\bar{\delta}_{\ell_r\oplus_p\ell_p},$$
    and Theorem \ref{GPanalyze} shows that the pair $(\ell_s\oplus_q\ell_q,\ell_r\oplus_p\ell_p)$ has property $(m)$. On the other hand,
    $$\bar{\delta}_{\ell_r\oplus_p\ell_p}=(1+t^r)^{\frac{1}{r}}-1>t-1.$$
    Consequently, by Corollary \ref{WmPmoduli}, $(\ell_s\oplus_q\ell_q,\ell_r\oplus_p\ell_p)$ has the WmP.
\end{proof}

It is natural to ask whether the WmP and the WMP are in fact equivalent. However, Example \ref{NewEx2} provides a negative answer to this question. Indeed, according to \cite[Theorem 2.11]{DJM}, for $1<p\leq s\leq q<\infty$, the pair $(\ell_s\oplus_q\ell_q,\ell_s\oplus_p\ell_p)$ does not satisfy the WMP, while it satisfies the WmP.

\begin{prop}
    For a Banach space $X$,
    \begin{enumerate}
        \item[(1)] The pair $(X,X)$ has property $(m)$ if and only if $X$ has property $(M)$.
        \item[(2)] The pair $(X,X)$ has property $(o)$ if and only if $X$ has the Opial property.
    \end{enumerate}
\end{prop}
\begin{proof}~
    For both assertions, the only if part follows by taking $T=\mathrm{id}_X$ in the corresponding definition of the pairwise property.
    
    Conversely, if $X$ has property $(M)$ (the Opial property, respectively), by the item (1) ((2), respectively) of Proposition \ref{BasicFacts}, the pair $(X,X)$ has property $(m)$ ((o), respectively).
\end{proof}

This observation highlights the relation between property $(M)$ and property $(m)$.
One may naturally ask whether these two properties are equivalent.
However, the pair $(\ell_s\oplus_q\ell_q,\ell_s\oplus_p\ell_p)$ considered above has property $(m)$ but does not have property $(M)$.
Indeed, the domain $\ell_s\oplus_q\ell_q$ is reflexive, and both $\ell_s\oplus_q\ell_q$ and $\ell_s\oplus_p\ell_p$ have the Opial property.
Since the pair $(\ell_s\oplus_q\ell_q,\ell_s\oplus_p\ell_p)$ does not satisfy the WMP, by the contrapositive of \cite[Theorem 12]{HK}, it follows that $(\ell_s\oplus_q\ell_q,\ell_s\oplus_p\ell_p)$ does not have property $(M)$. 

Finally, we conclude this section with a few immediate examples.

\begin{remark}\label{NewEx3}(\cite{Mil}, \cite[Example 2.2]{GP})
\begin{itemize}
    \item[(1)] For the James space $J$, $\bar{\delta}_J(t) = \bar{\rho}_J(t) = (1+t^2)^{\frac{1}{2}} - 1$ for all $t>0$.
    \item[(2)] If a Banach space $X$ admits a Schauder basis with a upper (lower, respectively) $\ell_p$-estimate of constant $1$, then $\bar{\rho}_X(t) \leq (1+t^p)^{\frac{1}{p}}-1$ ($\bar{\delta}_X(t) \geq (1+t^p)^{\frac{1}{p}}-1$, respectively). In particular, for $1<p<\infty$, if $Z=T^p$ or $Z=d(w,p)$, then $\bar{\rho}_Z(t) \leq (1+t^p)^{\frac{1}{p}}-1$, where $T^p$ is the $p$-convexification of the Tsirelson space \cite{Bra}, and $d(w,p)$ is the Lorentz sequence space \cite[p. 177]{LT}.
\end{itemize}
\end{remark}

Hence, we obtain the following examples.

\begin{example}
Let $Z$ be the $p$-convexification of the Tsirelson space $T^p$ or a Lorentz sequence space $d(w,p)$.
Then we have
\begin{itemize}
    \item[(a)] If $1< q\leq p<\infty$, then $(Z,\ell_q)$ has the WmP.
    \item[(b)] If $2\leq p<\infty$, then $(Z,J)$ has the WmP.
    \item[(c)] If $2\leq q<\infty$, then $(\ell_q,J)$ has the WmP.
\end{itemize}
\end{example}
\begin{proof}
    By Remark \ref{NewEx3} and Corollary \ref{WmPmoduli}, the above pairs satisfy both properties $(m)$ and $(o)$. 
    Moreover, since the Tsirelson space (see \cite{FJ}, for example) and the Lorentz sequence spaces (see \cite{LT}, for example) are reflexive, the conclusion follows from Theorem \ref{Main}.
\end{proof}

\section{Weak minimizing property in classical Banach spaces}\label{section3}

In this section, we provide several examples of pairs that either satisfy or fail to satisfy the WmP and the CPPm.
We particularly focus on pairs consisting of the classical sequence spaces $c_0$ and $\ell_p$ for $1 \leq p < \infty$.

The following example shows that none of these pairs satisfy the WmP.

\begin{example}\label{l1lp}
Let $1 \leq p < \infty$. Then
\begin{enumerate}
    \item The pairs $(\ell_1, \ell_p)$ and $(\ell_1, c_0)$ do not satisfy the WmP.
    \item The pair $(c_0, \ell_p)$ does not satisfy the WmP.
    \item The pair $(c_0, c_0)$ does not satisfy the WmP.
\end{enumerate}
\end{example}

\begin{proof}
Let $(e_i)_i$, $(f_i)_i$ and $(e_i^*)_i$ be the canonical bases of $c_0$, $\ell_p$ and $\ell_1$, respectively.
\begin{enumerate}
    \item Define an operator $T$ by $Te^*_n = 2^{-n}g_n$ for each $n\in\mathbb{N}$, where $(g_i)_i$ denotes the canonical basis of the range space.
    It is clear that $T$ does not attain its minimum modulus.
    But the sequence $(e_n^*)_n$ is a minimizing sequence for $T$ which does not converge weakly to zero.

    \item Define a sequence $(x_n^*)_n$ in $\ell_1$ by
    $$x_j^* = \frac{1}{2^j}e_j^* + \sum_{i=j+1}^\infty \frac{-1}{2^i}e_i^*
    \quad \text{for each } j \in \mathbb{N}.$$
    Consider the operator
    $$T = \sum_{i=1}^\infty x_i^* \otimes f_i.$$
    Then $T \in \mathcal{L}(c_0, \ell_p)$ and it satisfies
    \begin{align*}
        \left\| T\left(\sum_{k=1}^m e_k\right) \right\| &= \left\| \sum_{i=1}^\infty \sum_{k=1}^m x_i^*(e_k) f_i \right\| \\
        &= \left\| \frac{1}{2^m}f_m+\sum_{i=1}^{m-1} \left( \frac{1}{2^i} + \sum_{k=i+1}^m \frac{-1}{2^k} \right) f_i \right\| \\
        &= \left\|\frac{1}{2^m}f_m+\sum_{i=1}^{m-1}\frac{1}{2^m}f_i\right\| \\
        &= \frac{m^{\frac{1}{p}}}{2^m}.
    \end{align*}
    Therefore $m(T) = 0$, and the sequence $\left( \sum_{i=1}^n e_i \right)_n$ in $S_{c_0}$ is a non-weakly null minimizing sequence for $T$.
    Suppose that $x \in S_{c_0}$ satisfies $T(x) = 0$. 
    Then it is clear that $x^*_k(x)=0$, for every $k\in\mathbb{N}$.
    Writing $x = \sum_{i=1}^\infty x_i e_i$, we obtain
    $$\frac{1}{2^k}x_k-\sum_{i=k+1}^\infty \frac{1}{2^i}x_i=0,$$
    for every $k\in\mathbb{N}$.
    Hence,
    $$x_k=\sum_{i=1}^\infty\frac{1}{2^i}x_{k+i}.$$
    It follows, for each $k\in\mathbb{N}$ there exists $k_0>k$ such that $|x_{k_0}|\geq |x_k|$, which contradicts the fact that $x\in c_0$. 
    Consequently, $T$ does not attain its minimum modulus.

    \item Consider the operator $T \in \mathcal{L}(c_0)$ defined by
    $$T = 2 e_1^* \otimes e_1 + \sum_{i=2}^\infty \left( 1 + \frac{1}{i} \right) e_i^* \otimes e_i.$$
    Note that for any $x \in S_{c_0}$, there exists $n \in \mathbb{N}$ such that $|e_n^*(x)| = 1$. Therefore, $m(T)=1$, and the operator $T$ does not attain its minimum modulus. On the other hand, the sequence $\left( \frac{1}{2}e_1 + e_n \right)_n$ in $S_{c_0}$ is a non-weakly null minimizing sequence for $T$.
\end{enumerate}
\end{proof}

Note that for $1\leq p<\infty$, by Proposition \ref{BasicFacts}, the pairs $(\ell_1,\ell_p)$, $(\ell_1,c_0)$, $(c_0,\ell_1)$ and $(c_0,\ell_p)$ all satisfy both properties $(m)$ and $(o)$. 
However, Example \ref{l1lp} shows that these pairs do not have the WmP.
Hence, the reflexivity assumption in Theorem \ref{Main} cannot be omitted.

In \cite[Theorem 2.7]{DJM}, the authors proved that a pair $(X,Y)$ satisfies the WMP for every reflexive Banach space $X$ if and only if $Y$ has the Schur property. 
Recall that a Banach space is said to have the {Schur property} if every weakly convergent sequence also converges in norm to the same limit.
$\ell_1$ is a typical example of a space with the Schur property \cite{Sch}. 
We now establish one direction of the minimal modulus analogue of this result as follows.

\begin{prop}\label{Schurl1}
If $Y$ has the Schur property, then for any reflexive Banach space $X$, the pair $(X,Y)$ has the WmP. In particular, for any reflexive Banach space $X$, the pair $(X, \ell_1)$ has the WmP.
\end{prop}

\begin{proof}
Consider an operator $T\in\mathcal{L}(X,Y)$ and its non-weakly null minimizing sequence $(x_n)_n$. 
Take a subsequence $(x_{n_i})_i$ of $(x_n)_n$ which converges weakly to $x\in B_X$. 
Since $T$ is weak-weak continuous and $Y$ has the Schur property, it follows that $Tx_{n_i}\to Tx$, and hence $\|Tx\|=m(T)$.
\end{proof}

As mentioned in item (1) of Example \ref{l1lp}, Proposition \ref{Schurl1} cannot be extended to non-reflexive spaces.

A natural question is whether the pairs in Example \ref{l1lp} have the CPPm.
The answer is affirmative.
Indeed, if $m(T)=0$ for all $T\in\mathcal{L}(X,Y)$, then the pair $(X,Y)$ satisfies the CPPm, as follows from the observation below Remark \ref{Cha} in the case where $m(T)=0$.

On the other hand, these pairs also provide a negative answer to another natural question related to reflexivity.
Recall that a Banach space $X$ is reflexive if and only if there exists a Banach space $Y$ such that the pair $(X,Y)$ has the CPP (\cite{AGPT}, for example). 
However, by the previous observation, there exist several pairs $(X,Y)$ of Banach spaces satisfying the 
CPPm for which $X$ is not reflexive.

We proceed to give a few examples of pairs without the CPPm.

\begin{example}\label{l1the CPPm}
The pair $(\ell_1, \ell_1)$ does not have the CPPm.
\end{example}

\begin{proof}
Let us define operators $T \in \mathcal{L}(\ell_1)$ and $K \in \mathcal{K}(\ell_1)$ by
$$T = \sum_{i=2}^\infty 3e_i^* \otimes \left[\left(1-\frac{1}{2^i}\right)e_2 + e_{i+1}\right] - \frac{3}{2}e_1^* \otimes e_2$$
and
$$K = \frac{3}{2}e_1^* \otimes e_2 + 3e_1^* \otimes e_1,$$
where $(e_i)_i \subset \ell_1$ and $(e_i^*)_i \subset \ell_\infty$ are the canonical bases of $\ell_1$ and $\ell_\infty$, respectively.

For any $x = \sum_{i=1}^\infty \alpha_i e_i \in S_{\ell_1}$, we have
\begin{align*}
\|(T+K)x\| &= \left\|\sum_{i=1}^\infty \alpha_i (T+K)e_i\right\| \\
&= \left\| 3\left[\sum_{i=2}^\infty \alpha_i\left(1-\frac{1}{2^i}\right)e_2 + \sum_{i=2}^\infty \alpha_i e_{i+1}\right] + 3\alpha_1 e_1 \right\| \\
&= 3 + 3\left\|\sum_{i=2}^\infty \alpha_i\left(1-\frac{1}{2^i}\right) e_2 \right\| \\
&\geq 3.
\end{align*}
Thus, it follows that $m(T+K) \geq 3$.

Next, we compute $\|Tx\|$. We observe
\begin{align}\label{Txvalue}
\|Tx\| &= \left\|\sum_{i=1}^\infty \alpha_i Te_i\right\| \nonumber \\
&= 3\left\|-\frac{1}{2}\alpha_1 e_2 + \sum_{i=2}^\infty \alpha_i\left(1-\frac{1}{2^i}\right)e_2 + \sum_{i=2}^\infty \alpha_i e_{i+1}\right\| \nonumber \\
&= 3\sum_{i=2}^\infty |\alpha_i| + 3\left|\frac{1}{2}\alpha_1 - \sum_{i=2}^\infty \alpha_i\left(1-\frac{1}{2^i}\right)\right|.
\end{align}
For a sequence $(x_n)_n$ given by $x_n = \frac{2}{3}e_1 + \frac{1}{3}e_n$, we obtain
$$\|Tx_n\| = 1 + \frac{1}{2^n},$$
for each $n\in\mathbb{N}$, which implies $m(T) \leq 1$.
Hence 
$$\sup_{L\in\mathcal{K}(\ell_1)}m(T+L)>m(T).$$

Now we claim that $T$ does not attain its minimum modulus. Suppose, toward a contradiction, that there exists a sequence $(\beta_i)_i$ with $\sum_{i=1}^\infty |\beta_i| = 1 $ such that
$$\left\|T\left(\sum_{i=1}^\infty \beta_i e_i\right)\right\| = m(T).$$
Then by (\ref{Txvalue}), it is clear that $|\beta_1|<1$, and we have
$$\sum_{i=2}^\infty |\beta_i| \leq \frac{1}{3} \quad \text{and therefore,} \quad |\beta_1| \geq \frac{2}{3}.$$
In this case, it follows that
\begin{align*}
\left\|T\left(\sum_{i=1}^\infty \beta_i e_i\right)\right\| 
&= 3\left( \sum_{i=2}^\infty |\beta_i| + \left|\frac{1}{2}\beta_1 - \sum_{i=2}^\infty \beta_i\left(1-\frac{1}{2^i}\right)\right| \right) \\
&\geq 3\left( \sum_{i=2}^\infty |\beta_i| + \frac{1}{2}|\beta_1| - \sum_{i=2}^\infty |\beta_i|\left(1-\frac{1}{2^i}\right) \right) \\
&> \frac{3}{2}|\beta_1| \\
&\geq 1.
\end{align*}
This contradiction shows that $T$ does not attain its minimum modulus.
\end{proof}

In Section \ref{section2}, we presented examples of pairs of Banach spaces that satisfy the WmP but fail the WMP, and that have property $(m)$ but not property $(M)$.
We also get a result for $p$-sums in order to show both 
$$\text{WMP}\nRightarrow\text{WmP}\quad\text{and}\quad\text{property}~(M)\nRightarrow\text{property}~(m),$$
in general.

\begin{theorem}
Let $X$ and $Y$ be Banach spaces.
If there exists an operator $T\in\mathcal{L}(X,Y)$ with $m(T)>0$ that does not attain its minimum modulus, then for any $1\le p<q<\infty$ and any nontrivial finite-dimensional Banach space $Z$, the pair $(X\oplus_p Z,\,Y\oplus_q Z)$ fails to have the CPPm.
\end{theorem}

\begin{proof}
We may assume $m(T)=1$. Otherwise, replace $T$ by $\frac{T}{m(T)}$ so that $m(T)=1$.
Define $S\in\mathcal{L}(X\oplus_p Z,\,Y\oplus_q Z)$ by
$S(x,z)=(Tx,z)$. For $(x,z)\in X\oplus_p Z$, write $a=\|x\|$, $b=\|z\|$ so that
$a^p+b^p=1$ when $\|(x,z)\|_p=1$. Since $m(T)=1$, we have $\|Tx\|\ge a$. Hence
$$m(S)=\inf_{\|(x,z)\|_p=1}\|(Tx,z)\|_q
\geq\inf_{a^p+b^p=1}\left(a^q+b^q\right)^{\frac{1}{q}}.$$
Let $\lambda:=a^p\in[0,1]$. Then $a=\lambda^{\frac{1}{p}}$, $b=(1-\lambda)^{\frac{1}{p}}$ and
$$\left(a^q+b^q\right)^{\frac{1}{q}}
=\left(\lambda^{\frac{q}{p}}+(1-\lambda)^{\frac{q}{p}}\right)^{\frac{1}{q}}.$$
Since $q>p$, the function $\lambda\mapsto\lambda^{\frac{q}{p}}+(1-\lambda)^{\frac{q}{p}}$ is minimized
at $\lambda=\frac{1}{2}$, giving
$$\left(2\cdot 2^{-\frac{q}{p}}\right)^{\frac{1}{q}}=2^{\frac{p-q}{pq}}\leq m(S).$$
On the other hand, for a sequence $(x_n)_n$ in $S_X$ such that $\|Tx_n\|\to 1$, and an element $z\in S_Z$ consider a sequence $\left(2^{-\frac{1}{p}}x_n,2^{-\frac{1}{p}}z\right)_n$ of unit vectors. Then we have
$$\left\|S\left(\left(2^{-\frac{1}{p}}x_n,2^{-\frac{1}{p}}z\right)\right)\right\|_q=\left\|\left(2^{-\frac{1}{p}}Tx_n,2^{-\frac{1}{p}}z\right)\right\|_q=\left(2^{-\frac{q}{p}}\|Tx_n\|^q+2^{-\frac{q}{p}}\right)^{\frac{1}{q}}\to 2^{\frac{p-q}{pq}}.
$$
Thus $m(S)=2^{\frac{p-q}{pq}}<1=m(T)$.
Moreover, from above, we see that if $S$ attains its minimum modulus at some $(x,z)$ with $\|(x,z)\|_p=1$, then necessarily
$\|Tx\|=\|x\|$, which yields a unit vector $u=\frac{x}{\|x\|}$ with $\|Tu\|=1$.
This contradicts the assumption that $T$ does not attain $m(T)$.
Hence $S$ does not attain its minimum modulus.

Finally, note that $S=(T,0)+(0,\mathrm{id}_Z)$.
Therefore, if the pair $(X\oplus_p Z,\,Y\oplus_q Z)$ had the CPPm, then $S$ would attain its minimum modulus, which contradicts the conclusion of the previous paragraph.
Therefore, the pair fails the CPPm.
\end{proof}

\begin{example}\label{psums}
For $1\leq p<q<\infty$ and $1\leq r\leq s<\infty$, the pair $(\ell_r\oplus_p\mathbb{R}, \ell_s\oplus_q\mathbb{R})$ does not have the CPPm.
\end{example}

By \cite[Theorem 17]{HK}, for real numbers $1<p<q<\infty$ and $1<s<\infty$, the pair $(\ell_s\oplus_p\mathbb{R},\ell_s\oplus_q\mathbb{R})$ has property $(M)$.
In particular, it satisfies the WMP.
However, by Example \ref{psums}, this pair fails the CPPm. 
Consequently, it fails both property $(m)$ and the WmP.

We have now obtained nearly complete information on which pairs of separable classical sequence spaces satisfy the WmP or the CPPm, and which do not.
Nevertheless, despite our efforts, we were unable to decide whether the pair $(c_0,c_0)$ has the CPPm, and we conclude the paper by posing this as an open problem.

\begin{question}
    Does the pair $(c_0,c_0)$ have the CPPm?
\end{question}

\section*{Acknowledgements}

The author would like to thank Sun Kwang Kim for helpful discussions and valuable comments.

\subsection*{Funding}
The author was supported by the National Research Foundation of Korea(NRF) grant funded by the Korea government(MSIT) [NRF-2020R1C1C1A01012267].

\end{document}